\documentclass[12pt,reqno]{amsart}
\usepackage{amsmath}
\usepackage{amsthm}
\usepackage{verbatim}
\usepackage{stmaryrd}
\usepackage{amsfonts,mathrsfs,amssymb, bm}
\usepackage{times}
\usepackage[pdfencoding=auto]{hyperref}
\theoremstyle{plain}
\newcommand{\tensor}{\otimes}
\newtheorem{thm}{Theorem}[section]
\newtheorem{lem}[thm]{Lemma}

\theoremstyle{definition}

\newtheorem*{thm*}{Theorem}
\theoremstyle{remark}

\begin{document}
\title{The Demailly systems with the vortex ansatz}
\author{Arindam Mandal}
\address{Department of Mathematics, Indian Institute of Science, Bangalore 560012, India}
\email{arindamm@iisc.ac.in}
\maketitle
\begin{abstract}
    For an arbitrary-rank vector bundle over a projective manifold, J.-P. Demailly proposed several systems of equations of Hermitian-Yang-Mills type for the curvature tensor to settle a conjecture of Griffiths on the equivalence of Hartshorne ampleness and Griffiths positivity. In this article, we have studied two proposed systems and proved that these equations have smooth solutions for the Vortex bundle using the continuity method.
\end{abstract}
\let\thefootnote\relax
\section{Introduction}
Let $X$ be an $n$-dimensional projective manifold. A rank-$r$ holomorphic vector bundle $E$ over $X$ is said to be ample in the Hartshorne sense \cite{Har} if and only if the line bundle $\mathcal{O}_{\mathbb{P}(E)}(1)$ is ample over $\mathbb{P}(E)$. The Chern curvature tensor $\Theta_{E,h}$ of a Hermitian metric $h$ is said to be Griffiths positive if $\langle \sqrt{-1}\Theta_{E,h}(\zeta, \bar{\zeta}).v, v \rangle_h$ is positive for all decomposable nonzero elements $\zeta \tensor v \in T_X\tensor E$, and Nakano positive if the bilinear form on $T_X\tensor E$ defined by $\sqrt{-1}\Theta_{E,h}$ is positive. Nakano positivity and dual Nakano positivity (the bundle $(E^*,h^*)$ is Nakano negative) imply Griffiths positivity, which is equivalent to dual Griffiths positivity (the bundle $(E^*,h^*)$ is Griffiths negative). Griffiths positivity implies ampleness. B. Berndtsson \cite{Bo} has proved that for every positive integer $m$, $S^m E \tensor det E $ is Nakano positive if $E$ is ample. The tangent bundle $T{\mathbb{P}^n}$ of the complex projective space $\mathbb{P}^n$ is ample but not Nakano positive, and ampleness does not imply Nakano positivity (see \cite{Demailly} for details). A conjecture of Griffiths \cite{Gri} asks if Hartshorne ampleness implies Griffiths positivity. This conjecture is still open in its full generality. However, the conjecture holds for vector bundles on smooth curves, i.e., for $n=1$, see \cite{CaFl} and \cite{Um} for more details. Much work has been done in this direction (see \cite{LSY, MT, Na} and the references therein). J. P. Demailly  \cite{Demailly} introduced systems of PDE of Hermitian-Yang-Mills type for the curvature tensor to prove the equivalence between ampleness and Griffiths positivity. Let $(E,H_0)$ be smooth Hermitian holomorphic vector bundle of rank $r$ over $X$ such that $E$ is ample and $\omega_0=\sqrt{-1}\Theta_{det E, det H_0}>0$. Then one of Demailly's systems for time-dependent metrics $h_t, t \in [0,1]$ is as follows
\begin{equation}\label{dem1}
\begin{split}
& \omega_0^{-n}det_{T_X\tensor E^*} \Big( \sqrt{-1}\Theta_{E,h_t}+(1-t)\alpha \omega_0\tensor Id_{E^*} \Big)^{\frac{1}{r}}= \Big( \frac{det H_0}{det h_t} \Big)^{\lambda}a_0,\\
& \omega_0^{-n}\Big(\omega_0^{n-1} \wedge \sqrt{-1}(\Theta_{E,h_t}- \frac{1}{r}\Theta_{det E, det h_t}\tensor Id_E) \Big)\\
& \qquad \qquad =-\varepsilon \Big( \frac{det H_0}{det h_t} \Big)^{\mu}ln\Bigg( \frac{h_tH_0^{-1}}{det(h_tH_0^{-1})^{\frac{1}{r}}} \Bigg),
\end{split}
\end{equation}
where $\lambda \geq 0, \mu \in \mathbb{R}$, $\varepsilon >0$ and $a_0=\omega_0^{-n}det_{T_X\tensor E^*} \Big( \sqrt{-1}\Theta_{E,h_0}+\alpha \omega_0\tensor Id_{E^*} \Big)^{\frac{1}{r}}>0$. The metric $h_0$ is a solution of the second equation (cushioned Hermite-Einstein equation) of system \eqref{dem1} at $t=0$ with the condition $det h_0= det H_0$. With the notations same as above, one more variant of the above system is
\begin{equation}\label{dem2}
\begin{split}
&\omega_0^{-n}det_{T_X\tensor E^*} \Big( \sqrt{-1}\Theta_{E,h_t}+(1-t)\alpha \omega_0\tensor Id_{E^*} \Big)^{\frac{1}{r}}= \Big( \frac{det H_0}{det h_t} \Big)^{\lambda}a_0,\\
& \omega_0^{-n}\Big(\omega_t^{n-1} \wedge \sqrt{-1}(\Theta_{E,h_t}- \frac{1}{r}\Theta_{det E, det h_t}\tensor Id_E) \Big)\\
& \qquad \qquad =-\varepsilon \Big( \frac{det H_0}{det h_t} \Big)^{\mu}ln\Bigg( \frac{h_tH_0^{-1}}{det(h_tH_0^{-1})^{\frac{1}{r}}} \Bigg),
\end{split}
\end{equation}
where $\omega_t=\frac{1}{1+r\alpha}(\sqrt{-1}\Theta_{det E, det h_t}+r(1-t)\alpha \omega_0)$ and the metric $h_0$ is a solution of second equation of the system \eqref{dem2} satisfying $det h_0= det H_0$. Even though for both of the above systems, the existence of the metric $h_0$ is clear from \cite{UY86}, for our purpose, in the case of Vortex bundle, we shall discuss the existence separately (subsections \ref{initial1} and \ref{initial2}). If one can prove the existence of $h_t$ for all time $t\in [0,1]$, for one of the above systems, then $h_t$ will be dual Nakano positive for $t=1$. Thus a stronger result than Griffiths conjecture will have been proven, which will certainly settle the Griffiths conjecture. It turns out that there exists ample bundle which is not dual Nakano positive; see \cite{Demailly} for a detailed example. Therefore one should not expect the existence of solutions of Demailly's systems for all time $t\in [0,1]$ in general.  V. P. Pingali \cite{directsum} has studied the system \eqref{dem1} for the direct sum of ample line bundles on Riemann surfaces using Leray-Schauder degree theory. In this article, using the continuity method, we have studied systems \eqref{dem1} and \ref{dem2} on the Vortex bundle.
\\

We follow the construction of Vortex bundles as in \cite{Prada} and \cite{Vec}. Let $\Sigma$ be a compact Riemann surface with a background Hermitian metric $k$ on an ample holomorphic line bundle $L$ such that $\omega_{\Sigma}=\sqrt{-1} \Theta_k$ is the K\"ahler metric, where $\Theta_k$ is the curvature of the metric $k$. Consider $\mathbb{CP}^1$ with the metric $h_{FS}$ on the line bundle $\mathcal{O}(1)$, who's curvature is the Fubini-study metric $\omega_{FS}=\frac{\sqrt{-1} dz \wedge d\bar{z}}{(1+|z|^2)^2}$. Define the rank-2 vector bundle $E$ on the projective manifold $X=\Sigma \times \mathbb{CP}^1$ by $E=\pi_{1}^*((r_1+1)L) \tensor \pi_{2}^*(r_2\mathcal{O}(2)) \oplus \pi_{1}^*(r_1L) \tensor \pi_{2}^*((r_2+1)\mathcal{O}(2))$, where $\pi_1 : X \rightarrow \Sigma $ and $\pi_2 : X \rightarrow \mathbb{CP}^1$ are the projection maps and $r_1, r_2$ are positive integers. Let $\phi \in H^0(\Sigma, L)$ be a global holomorphic section. We see that multiplication of the metric $k$ by a nonzero constant does not change the curvature $\Theta_k$. So we assume $k$ has been rescaled so that $|\phi|_k^2 \leq \frac{1}{2}$. Define a holomorphic structure on $E$ by the second fundamental form $\beta=\pi_1^*\phi \tensor \pi_2^* \Big(\frac{\sqrt{8\pi}dz}{(1+|z|^2)^2}\tensor d\bar{z} \Big)$. Let $E$ be equipped with metric $h_t=\pi_{1}^*(e^{-(f_t+\psi_t)}k^{r_1+1}) \tensor \pi_{2}^*(h_{FS}^{2r_2}) \oplus \pi_{1}^*(e^{-f_t}k^{r_1}) \tensor \pi_{2}^*(h_{FS}^{2r_2+2})$, where $f_t$ and $\psi_t$ are smooth function on $\Sigma$. Suppose $\Tilde{h}_t=\pi_{1}^*(e^{-(f_t+\psi_t)}k^{r_1+1}) \tensor \pi_{2}^*(h_{FS}^{2r_2})$, $\Tilde{g}_t=\pi_{1}^*(e^{-f_t}k^{r_1}) \tensor \pi_{2}^*(h_{FS}^{2r_2+2})$ and $g_t=e^{-\psi_t}k$. Then the Chern connection of $(E,h_t)$ for the holomorphic structure given by $\beta$ is given by the following connection matrix 
\begin{equation*}
A_{h_t}=
\begin{bmatrix}
 A_{\Tilde{h}_t} &\beta \\
 -\beta^{\dag_{g_t}} & A_{\Tilde{g}_t}
\end{bmatrix}.
\end{equation*}
Its curvature matrix is
\begin{equation*}
\Theta_{h_t}=
\begin{bmatrix}
 \Theta_{\Tilde{h}_t}-\beta \wedge \beta^{\dag_{g_t}} &\nabla^{1,0}\beta \\
 -\nabla^{0,1}\beta^{\dag_{g_t}} & \Theta_{\Tilde{g}_t}-\beta^{\dag_{g_t}}\wedge \beta
\end{bmatrix}.
\end{equation*}
Let $H_0=\pi_{1}^*(k^{r_1+1}) \tensor \pi_{2}^*(h_{FS}^{2r_2}) \oplus \pi_{1}^*(k^{r_1}) \tensor \pi_{2}^*(h_{FS}^{2r_2+2})$ be the background metric on the Vortex bundle $E$. Then $\omega_0=(2r_1+1)\omega_{\Sigma}+(4r_2+2)\omega_{FS}$. Choosing $\lambda=0, \mu=0$, and $\varepsilon=1$, the system \eqref{dem1} for the Vortex bundle will be the following decoupled system of equations.
\begin{equation}\label{01}
\begin{split}
 & \Big\{ \Big( \Delta f_t +\Delta {\psi}_t +(r_1 +1)+ \alpha (1-t)(2r_1+1)\Big) \Big( 2r_2 + |\phi|_{g_t}^2 +  \\
 & \alpha (1-t)(4r_2+2) \Big) \Big( \Delta f_t + r_1+ \alpha (1-t)(2r_1+1)\Big)  \\
 & \Big( (2r_2+2)- |\phi|_{g_t}^2 + \alpha (1-t)(4r_2+2) \Big)\Big\} + \sqrt{-1} \frac{\nabla^{1,0} \phi \wedge \nabla^{0,1} \phi^{\dag_{g_t}}}{\omega_{\Sigma}}\\
 & \Big( 2r_2 + |\phi|_{g_t}^2 + \alpha (1-t)(4r_2+2) \Big)\Big( \Delta f_t + r_1+ \alpha (1-t)(2r_1+1)\Big)\\
 = & \Big\{ \Big(\Delta f_0 + \Delta \psi_0+ (r_1 + 1 ) +\alpha (2r_1+1) \Big)\Big( 2r_2+ |\phi|_{g_0}^2 + \alpha (4r_2+2) \Big)  \\
 & \Big( (2r_2+2)- |\phi|_{g_0}^2 + \alpha (4r_2+2) \Big) \Big( \Delta f_0 + r_1+ \alpha (2r_1+1)\Big)\Big\}\\
  &  + \sqrt{-1} \frac{\nabla^{1,0} \phi \wedge \nabla^{0,1}\phi^{\dag_{g_0}}}{\omega_{\Sigma}}\Big( 2r_2 + |\phi|_{g_0}^2 + \alpha (4r_2+2) \Big)\\
  & \Big(\Delta f_0 + r_1+ \alpha (2r_1+1)\Big),
 \end{split}
\end{equation}

\begin{equation}\label{02}
 (2r_1+1)\big( |\phi|_{g_t}^2 -1 \big) +\big( \Delta \psi_t +1 \big)(2r_2+1)= (2r_1 + 1)(4r_2 + 2) \psi_t,
\end{equation}
where $h_0=\pi_{1}^*(e^{-(f_0 + \psi_0)}k^{r_1+1}) \tensor \pi_{2}^*(h_{FS}^{2r_2}) \oplus \pi_{1}^*(e^{-f_0}k^{r_1}) \tensor \pi_{2}^*(h_{FS}^{2r_2+2})$ is the solution of the equation \eqref{02} at $t=0$, satisfying $det h_0= det H_0$. The existence of such $h_0$ is discussed in the subsection \ref{initial1} and $\alpha >0$ is a large enough constant so that $\sqrt{-1} \Theta_{h_0} + \alpha \omega_{0} \tensor Id_{E^*} >0$ in the sense of Nakano and $\Delta f_0 + r_1 + \alpha (2r_1+1)>0$. We now state one of our results.
 \begin{thm}\label{main1}
 The system defined by equations \eqref{01} and \eqref{02} has smooth solution $(f_t, \psi_t)$ such that $\Delta f_t + r_1 + \alpha (1-t)(2r_1+1) >0$  for all $t\in [0,1].$
 \end{thm}
If we choose $\lambda=0$ and $\mu=0$, then the system \eqref{dem2} for the Vortex bundle will be the following coupled system of equations.
\begin{equation}\label{1}
\begin{split}
&\Big\{ \Big( \Delta f_t +\Delta {\psi}_t +(r_1+1)+ \alpha (1-t)(2r_1+1)\Big) \Big( 2r_2 + |\phi|_{g_t}^2 +  \\
& \alpha (1-t)(4r_2+2) \Big) \Big( \Delta f_t +  r_1+ \alpha (1-t)(2r_1+1)\Big) \\
& \Big( (2r_2+2)- |\phi|_{g_t}^2 + \alpha (1-t)(4r_2+2) \Big)\Big\} \\
& + \sqrt{-1} \frac{\nabla^{1,0} \phi \wedge \nabla^{0,1} \phi^{\dag_{g_t}}}{\omega_{\Sigma}} \Big( 2r_2 + |\phi|_{g_t}^2 + \alpha (1-t)(4r_2+2) \Big) \\
 & \Big( \Delta f_t + r_1+ \alpha (1-t)(2r_1+1)\Big)\\
  = &\Big\{ \Big(\Delta f_0 + \Delta \psi_0+ (r_1 + 1 ) + \alpha (2r_1+1) \Big)\Big( 2r_2+ |\phi|_{g_0}^2 + \alpha (4r_2+2) \Big) \\
 & \Big( (2r_2+2)- |\phi|_{g_0}^2 + \alpha (4r_2+2) \Big) \Big( \Delta f_0 + r_1+ \alpha (2r_1+1)\Big) \Big\}  \\
 & + \sqrt{-1} \frac{\nabla^{1,0} \phi \wedge \nabla^{0,1}\phi^{\dag_{g_0}}}{\omega_{\Sigma}}\Big( 2r_2 + |\phi|_{g_0}^2 + \alpha (4r_2+2) \Big) \\ 
 & \Big( \Delta f_0 + r_1+ \alpha (2r_1+1)\Big),
 \end{split}
\end{equation}
\begin{equation}\label{2}
\begin{split}
&\Big\{ 2\Big( 2\Delta f_t +\Delta \psi_t +(1+ 2\alpha (1-t))(2r_1+1) \Big)\big( |\phi|_{g_t}^2 -1 \big)\\
& +\big( \Delta \psi_t +1 \big)\big(1+ 2\alpha (1-t)\big)(4r_2+2)\Big\}\\
 &\qquad \qquad =2(2r_1 + 1)(4r_2 + 2)(2\alpha +1) \varepsilon \psi_t,
 \end{split}
\end{equation}
where $h_0=\pi_{1}^*(e^{-(f_0 + \psi_0)}k^{r_1+1}) \tensor \pi_{2}^*(h_{FS}^{2r_2}) \oplus \pi_{1}^*(e^{-f_0}k^{r_1}) \tensor \pi_{2}^*(h_{FS}^{2r_2+2})$ is the solution of the equation \eqref{2} at $t=0$, satisfying $det h_0= det H_0$. The existence of such $h_0$ is discussed in subsection \ref{initial2} and $\alpha >0$ is a large enough constant so that $\sqrt{-1}\Theta_{h_0} + \alpha \omega_0 \tensor Id_{E^*} >0$ in the sense of Nakano and $\Delta f_0 + r_1 + \alpha (2r_1+1)>0$. Finally, we have the following result.

\begin{thm}\label{main2}
 For large enough $\varepsilon>0$, the system defined by equations \eqref{1} and \eqref{2} has smooth solution $(f_t, \psi_t)$ such that $\Delta f_t + r_1 + \alpha (1-t)(2r_1+1) >0$ for all $t\in [0,1].$
 \end{thm}

\indent Demailly's original approach involved the method of continuity. However, proving openness for Demailly's systems is the most challenging part because the required positivity properties for openness may not be preserved along the continuity path. Even in the case of a direct sum of ample line bundles on a Riemann surface, it appears hard to prove, and therefore the Leray-Schauder degree method was used in \cite{directsum}. The main point of this article is to provide a proof-of-concept for Demailly's approaches. We hope that the techniques used for the vortex bundle generalize to more complicated situations. \\

\indent We briefly describe the strategy of the proofs. System \eqref{dem1} is decoupled and hence is relatively easier to handle (Section \ref{sec:first}). On the other hand, unlike System \eqref{dem1}, System \eqref{dem2} is truly a coupled system. To demonstrate openness, the key point is to prove the lower bound for $ \Delta \psi_t$, independent of $t$ and $\varepsilon$ so that $\Delta \psi_t + 1 + 2(2r_1+1)(4r_2+2)(2\alpha + 1) \varepsilon$ can be made positive for large $\varepsilon$. However, as one will see, to get such estimates, it is crucial to observe that the lower bound of $\Delta \psi_0$ itself is independent of $\varepsilon$. These calculations are rather delicate and carried out in Section \ref{sec:second}.
\section{Proof of Theorem \ref{main1} }\label{sec:first}
For the remainder of the paper, we drop the parameter $t$ for notational convenience. We denote constants by $C$ that may vary from line to line and are independent of $t$ unless specified.

\subsection{  \texorpdfstring{Existence of solution at $t=0$ for the first system}{\texttwoinferior} }\label{initial1}
Recall that $h_0 = \pi_{1}^*(e^{-(f_0 + \psi_0)}k^{r_1+1}) \tensor \pi_{2}^*(h_{FS}^{2r_2}) \oplus \pi_{1}^*(e^{-f_0}k^{r_1}) \tensor \pi_{2}^*(h_{FS}^{2r_2+2})$ is the solution of the equation \eqref{02} at $t=0$, and $det h_0= det H_0$. Therefore, $2f_0 + \psi_0=0$ and $\psi_0$ satisfying $(2r_1+1)\big( |\phi|_{g_0}^2 -1 \big) +\big( \Delta \psi_0 +1 \big)(2r_2+1)= (2r_1 + 1)(4r_2 + 2) \psi_0$. We shall show that $\psi_0$ exists by the method of continuity. Now define $L_s(\psi_0)=\Delta \psi_0 + 1 - s(1-|\phi|_{g_0}^2)\frac{2r_1+1}{2r_2+1}-2(2r_1+1)\psi_0$, where $s \in [0,1]$. Let $S:=\{ s \in [0,1] | L_s=0 \text{ has a smooth solution at } s \}$. Clearly $\psi_0=\frac{1}{2(2r_1+1)}$ is a solution of $L_0=0$. Thus $0 \in S$. Now $DL_s(\psi_0)[\delta \psi]=\Delta \delta \psi - s|\phi|^2_{g_0}\frac{2r_1+1}{2r_1+1}\delta \psi - 2(2r_1+1)\delta \psi$. By maximum principle we get $Ker(DL_s)=\{ 0\}$. Since $DL_s$ is formally self-adjoint, we get $DL_s$ is an isomorphism. Hence by implicit function theorem for Banach manifolds, we get $S$ is open. Now we must prove that $S$ is closed. Let us first prove a prior bound on the term $|\phi|_{g_0}^2$.
\begin{lem}\label{i2}
    $|\phi|_{g_0}^2<1$.
\end{lem}
\begin{proof}
     We know
\begin{equation}\label{05}
\nonumber \partial \bar{\partial}|\phi|_{g_0}^2=-\Theta_{g_0}|\phi|^2_{g_0}+\nabla^{1,0} \phi \wedge \nabla^{0,1} \phi^{\dag_{g_0}}.
\end{equation}
At the point $p$ of maximum of $|\phi|_{g_0}^2$, $\sqrt{-1} \partial \bar{\partial}|\phi|_{g_0}^2(p) \leq 0$. Thus we see that $\sqrt{-1} \Theta_{g_0}(p) \geq 0$, which implies $\big(1+ \Delta \psi_0 \big) (p) \geq 0$. If $\psi_0(p) \geq 0$, we have $|\phi|^2_{g_0}(p)=|\phi|^2_{k}(p)e^{-\psi_0(p)} \leq |\phi|^2_{k}(p) \leq \frac{1}{2}$. Otherwise $\psi_0(p)<0$, and equation \eqref{02} implies $(2r_1+1)\big( |\phi|_{g_0}^2 -1 \big)<0$. Hence, We get $|\phi|^2_{g_0}<1$.
\end{proof}
Let $s \in S$, then $\Delta \psi_0 + 1 = s(1-|\phi|_{g_0}^2)\frac{2r_1+1}{2r_2+1}-2(2r_1+1)\psi_0$. Applying maximum principle and Lemma \ref{i2} we get $||\psi_0||_{C^0}<C$ and hence $||\Delta \psi_0||_{C^0}<C$. Now by the Arzela-Ascoli theorem, we see that $S$ is closed. This completes the proof of the existence of $h_0$.

\begin{proof}[Theorem \ref{main1}]
Since the system is decoupled, we shall solve the equation \eqref{02} for $\psi$ using the continuity method. Then we shall solve equation \eqref{01} using the same method. Let 
\begin{align*}
    I^{\prime}_1 :=  \Big\{ t \in [0,1] \Big| \text{ equation \eqref{02} has smooth solution $\psi_t$ at $t$ } \Big\}.
\end{align*}
Since $\psi_0$ solves the equation at $t=0$, so $0 \in I_1^{\prime}$. Hence $I^{\prime}_1$ is non-empty. We want to show $I^{\prime}_1$ is closed and open. Then by connectedness  we can conclude that $I^{\prime}_1=[0,1]$.\\
 
\subsection{ Closedness of \texorpdfstring{$I^{\prime}_1$}{\texttwoinferior} }
If we prove suitable uniform estimates for $\psi_t$ whenever $t \in I^{\prime}_1$, then using the Arzela-Ascoli theorem we can prove that $I^{\prime}_1$ is closed.\\

 A similar calculation as in Lemma \ref{i2} shows the following estimates.

\begin{lem}\label{04}
 $|\phi|^2_{g} <1 $  for all $t \in [0,1].$
\end{lem}
Now let us prove $C^0$-estimates for $\psi_t$.
\begin{lem}\label{06}
If $ t \in I_1^{\prime}$, then there exists a constant $C$ such that $||\psi_t||_{C^0} \leq C$.
\end{lem}
\begin{proof}
Suppose $\psi$ attains its maximum at a point $p$, then $\Delta \psi(p) \leq 0$. Using this along with equation \eqref{02}, we conclude the following inequality 
\begin{eqnarray}\label{07}
(2r_1 + 1)(4r_2 + 2) \psi (p)\leq (2r_2+1).
\end{eqnarray}
At a minimum point $q$ for $\psi$ we have $\Delta \psi (q)\geq 0.$ Equation $\eqref{02}$ gives
\begin{eqnarray}\label{08}
  && -(2r_1+1) \leq (2r_1 + 1)(4r_2 + 2) \psi(q).
\end{eqnarray}
\end{proof}
Applying Lemma \ref{06} in equation \eqref{02}, we have the following estimate.
\begin{lem}\label{09}
Any smooth $\psi$ that solves the equation (\ref{02}) satisfies $|| \Delta \psi ||_{C^0} \leq C$ for some positive constant $C$.
\end{lem}
\subsection{Openness of \texorpdfstring{$I^{\prime}_1$}{\texttwoinferior}}
Let us define the map $T_2: C^{2,\beta} \times [0,1] \rightarrow  C^{0,\beta}$  by $T_2\Big( \psi , t \Big)= (2r_1+1)\big( |\phi|_{g}^2 -1 \big) +\big( \Delta \psi +1 \big)(2r_2+1) - (2r_1 + 1)(4r_2 + 2) \psi$. Its linearization is
\begin{equation}\label{010}
\begin{split}
& DT_2(\psi, t)[\delta \psi]\\ 
& = (2r_2+1)\Delta \delta \psi -(2r_1+1)|\phi|^2_{g} \delta \psi -(2r_1+1)(4r_2+2)\delta \psi.
\end{split}
\end{equation}
If $\delta \psi \in Ker(DT_2)$, then $DT_2(\psi)[\delta \psi]=0$. Now maximum principle implies that $\delta \psi =0$. Since $T_2$ is self-adjoint and $Ker(DT_2)=0$, it is an isomorphism.\\
 From Lemma \ref{06}, Lemma \ref{09}, and by bootstrapping, we conclude that solutions are smooth. Therefore, we have uniform estimates of solutions of the equation \eqref{02} and its derivative of all orders. \\
 
 \indent Now we can solve the equation \eqref{01} for the variable $f$. As we mentioned earlier, we shall do so by the continuity method.
We let 
\begin{align*}
     &&I_1^{\prime \prime} := \Big\{ t \in [0,1]  \text{ \Big| equation (\ref{01}) has smooth  solution $ f_t $ at $t$ and }\\
    && \textbf{ $\Delta f_t + r_1 + \alpha (1-t)(2r_1+1) >0$ } \Big\}.
\end{align*}
At $t=0$, $f_0$ is a solution of equation \eqref{01} and $\Delta f_0 + r_1 + \alpha (1-t)(2r_1+1) >0$, so $0 \in I_1$. Hence $I_1$ is non-empty. We must show that $I_1^{\prime \prime}$ is both open and closed.

\subsection{ Closedness of \texorpdfstring{$I^{\prime \prime}_1$}{\texttwoinferior}}
Let us prove the closedness of $I_1^{\prime \prime}$ by proving estimates for $f$ and its derivatives.\\
 
\begin{lem}\label{011}
There exists a constant $C$ such that whenever $t \in I^{\prime \prime}_1$, we have $||\Delta f_t||_{C^0} \leq C$.
\end{lem}

\begin{proof}
Since $t \in I^{\prime \prime}_1$, we have $\Delta f > -\big(r_1+\alpha(2r_1+1)\big)$.\\
Now from equation \eqref{01} we compute,
\begin{equation}\label{012}
\begin{split}
 & \Big\{ \Big( \Delta f +\Delta \psi +(r_1+1)+ \alpha (1-t)(2r_1+1)\Big) \Big( 2r_2 + |\phi|_{g}^2   \\
  & + \alpha (1-t)(4r_2+2) \Big) \Big( \Delta f + r_1+ \alpha (1-t)(2r_1+1)\Big)\\
  & \Big( (2r_2+2)- |\phi|_{g}^2 + \alpha (1-t)(4r_2+2) \Big)\Big\} \\
 & \leq \Big( \Delta f_0 + \Delta \psi_0 + (r_1 + 1 ) +\alpha (2r_1+1) \Big) \Big( 2r_2+ |\phi|_{g_0}^2 + \alpha (4r_2+2) \Big) \\
&\Big( (2r_2+2)- |\phi|_{g_0}^2 + \alpha (4r_2+2) \Big) \Big( \Delta f_0 + r_1+ \alpha (2r_1+1)\Big) \\
& + \sqrt{-1}\frac{\nabla^{1,0} \phi \wedge \nabla^{0,1}\phi^{\dag_{g_0}}}{\omega_{\Sigma}} \Big( 2r_2 + |\phi|_{g_0}^2 + \alpha (4r_2+2) \Big)\\
& \Big( \Delta f_0 + r_1+ \alpha (2r_1+1)\Big).
\end{split}
\end{equation}
Now appealing to Lemma \ref{09}, we get $ \Delta f \leq C$.
\end{proof}
Let $G$ be the Green's function of the metric $\omega_{\Sigma}$ such that $ - C \{ 1 + |ln(d_{\omega_{\Sigma}}(x,y))|\} \leq G(x,y) \leq 0$. Then for any continuous function $f$, we have the following Green representation formula:
\begin{eqnarray}\label{013}
  f(x)=\displaystyle \frac{\int_{\Sigma} f(y)\omega_{\Sigma}(y)}{\int_{\Sigma} \omega_{\Sigma}(y)} + \int_{\Sigma} G(x,y) \Delta f(y) \omega_{\Sigma}(y)
\end{eqnarray}
Using the formula \eqref{013} and Lemma \ref{011} we have the following:
\begin{lem}
If $t \in I^{\prime \prime}_1$ then $||f_t||_{C^0}\leq C$, for some positive constant $C$.
\end{lem}
Let $t_n \in I^{\prime \prime}_1$ be such that $t_n \rightarrow t$. To prove $I^{\prime \prime}_1$ is closed, we must show $t \in I^{\prime \prime}_1$. As $t_n \in I^{\prime \prime}_1$, we have $ \Delta f_{t_n}+r_1+\alpha (1-t)(2r_1+1)>0$. Arzela-Ascoli theorem, together with the above estimates, we get a subsequence of $t_n$ again call the subsequence by $t_n$ such that $f_{t_n} \rightarrow f$ in $C^{2,\alpha}$. Then by usual bootstrapping argument we have $f$ is smooth and $,\Delta f + r_1 + \alpha (1-t)(2r_1+1) \geq 0$. If $\Delta f + r_1 + \alpha (1-t)(2r_1+1) = 0$ at some point, then equation \eqref{01} gives a contradiction. Hence $\Delta f + r_1 + \alpha (1-t)(2r_1+1) >0$. Therefore $t \in I_1^{\prime \prime}$.\\

We now proceed to prove the openness of $I_1^{\prime \prime}$.

\subsection{Openness of \texorpdfstring{$I^{\prime \prime}_1$}{\texttwoinferior}}
Let $\mathcal{B}$ be the subset in $C^{2, \beta}$ defined by $\mathcal{B}:= \Big\{ f\in C^{2, \beta} \Big| \Delta f + r_1 + \alpha (1-t)(2r_1+1)>0, \displaystyle \int_{\Sigma} f \omega_{\Sigma}=0 \Big\}$. Now let us define the map $T_{1}: \mathcal{B} \times [0,1] \rightarrow   C^{0,\beta} $ by
\begin{equation}
\begin{split}
 & T_{1}(f, t) \\
= & \bigg\{ \Big( \Delta f +\Delta {\psi} +(r_1+1)+ \alpha (1-t)(2r_1+1)\Big) \Big( 2r_2 + |\phi|_{g}^2 + \\
& \alpha (1-t)(4r_2+2) \Big) \Big( \Delta f + r_1+ \alpha (1-t)(2r_1+1)\Big)\\
& \Big( (2r_2+2)- |\phi|_{g}^2 + \alpha (1-t)(4r_2+2) \Big) \\
& + \sqrt{-1} \frac{\nabla^{1,0} \phi \wedge \nabla^{0,1} \phi^{\dag_{g}}}{\omega_{\Sigma}} \Big( 2r_2 + |\phi|_{g}^2 + \alpha (1-t)(4r_2+2) \Big)\\
& \Big( \Delta f + r_1+ \alpha (1-t)(2r_1+1)\Big) \bigg\} - \bigg\{ \Big( \Delta f_0 + \Delta \psi_0 + (r_1 + 1 )  +\alpha (2r_1+1) \Big)\\
& \Big( 2r_2+ |\phi|_{g_0}^2 + \alpha (4r_2+2) \Big) \Big( (2r_2+2)- |\phi|_{g_0}^2 + \alpha (4r_2+2) \Big)\\
& \Big( \Delta f_0 + r_1+ \alpha (2r_1+1)\Big) + \sqrt{-1} \frac{\nabla^{1,0} \phi \wedge \nabla^{0,1}\phi^{\dag_{g_0}}}{\omega_{\Sigma}} \Big( \Delta f_0 + r_1+ \alpha (2r_1+1)\Big)\\
& \Big( 2r_2 + |\phi|_{g_0}^2 + \alpha (4r_2+2) \Big) \bigg\}.
 \end{split}
\end{equation}
Then its linearization at a point $(f,t) \in \mathcal{B}\times [0,1]$ will be  
\begin{equation}
\begin{split}
& DT_{1}(f,t)[\delta f]\\
= & \Big( 2r_2 + |\phi|_{g}^2 + \alpha (1-t)(4r_2+2) \Big)\bigg[\Big( (2r_2+2)- |\phi|_{g}^2 + \\
&\alpha (1-t)(4r_2+2)\Big) \Big( 2\Delta f+\Delta \psi +\big(1+ 2\alpha (1-t)\big)(2r_1+1) \Big) \\
& +\sqrt{-1}\frac{\nabla^{1,0} \phi \wedge \nabla^{0,1} \phi^{\dag_{g}}}{\omega_{\Sigma}}\bigg]\Delta \delta f.
\end{split}
\end{equation}
 If $t \in I^{\prime \prime}_1$, then clearly $DT_1(f,t)$ is an isomorphism. Therefore by the implicit function theorem on Banach manifolds, we get $I^{\prime \prime}_1$ is open. This completes the proof of the Theorem (\ref{main1}).
 \end{proof}
\section{Proof of theorem \ref{main2}}\label{sec:second}
Our next concern is to prove Theorem (\ref{main2}). First, let us discuss the solution of the second system at $t=0$. 
\subsection{  \texorpdfstring{Exixtence of solution at $t=0$ for the second system}{\texttwoinferior} }\label{initial2}
Recall that $h_0 = \pi_{1}^*(e^{-(f_0 + \psi_0)}k^{r_1+1}) \tensor \pi_{2}^*(h_{FS}^{2r_2}) \oplus \pi_{1}^*(e^{-f_0}k^{r_1}) \tensor \pi_{2}^*(h_{FS}^{2r_2+2})$ is the solution of the equation \eqref{2} at $t=0$, and $det h_0= det H_0$. Therefore $(f_0, \psi_0)$ satisfying $2f_0 + \psi_0=0$ and 
\begin{equation*}
\begin{split}
     &\Big\{ 2\Big( 2\Delta f_0 +\Delta \psi_0 +(1+ 2\alpha )(2r_1+1) \Big)\big( |\phi|_{g_0}^2 -1 \big) +\big( \Delta \psi_0 +1 \big)\big(1+ 2\alpha \big)(4r_2+2)\Big\}\\
    &=2(2r_1 + 1)(4r_2 + 2)(2\alpha +1) \varepsilon \psi_0.
\end{split}
\end{equation*}
Equivalently, we shall solve the following equation for $\psi_0$,
\begin{align*}
    \Delta \psi_0 + 1 = (1-|\phi|_{g_0}^2)\frac{2r_1+1}{2r_2+1}+2(2r_1+1)\varepsilon \psi_0.
\end{align*}
Similar arguments in subsection \ref{initial1} prove the existence of required $\psi_0$. Moreover, we get the following estimates independent of $\varepsilon$.
\begin{lem}\label{i5}
    There exists positive constant $C$ independent of $\varepsilon$ such that $||\varepsilon \psi_0||_{C^0} + ||\Delta \psi_0||_{C^0}<C$.
\end{lem}
 Let 
\begin{align*}
    &&I_2 :=  \Big\{ t \in [0,1] \Big| \text{ the system defined by equations \eqref{1} and \eqref{2} has smooth}\\
    &&\text{  solution $(f_t, \psi_t)$ at $t$ }, \Delta f_t + r_1+ \alpha (1-t)(2r_1+1) >0 \Big\}.
\end{align*}
At $t=0$, $(f_0 ,\psi_0)$ solves the system and $\Delta f_0 + r_1+ \alpha (2r_1+1) >0$. So $I_2$ is non-empty. Now let us prove $I_2$ is closed by proving some a priori estimates.
\subsection{ Closedness of \texorpdfstring{$I_2$}{\texttwoinferior} }
A similar proof as Lemma \ref{i2} gives the following result.
\begin{lem}\label{4}
$|\phi|^2_{g_t} <1 $  for $t \in I_2.$
\end{lem}
\begin{proof}
    A simple computation using normal coordinates gives the following identity
\begin{equation}\label{33}
\partial \bar{\partial}|\phi|_{g_t}^2=-\Theta_{g_t}|\phi|^2_{g_t}+\nabla^{1,0} \phi \wedge \nabla^{0,1} \phi^{\dag_{g_t}}.
\end{equation}
Let $|\phi|^2_{g_t}$ attains its maximum at a point $p$, then $\sqrt{-1} \partial \bar{\partial}|\phi|_{g_t}^2(p) \leq 0$. So we must have $\sqrt{-1}\Theta_{g_t}(p) \geq 0$ i.e., $(1+\Delta \psi_t)(p) \geq 0$. If $\psi_t(p) \geq 0$, then $|\phi|^2_{g_t}(p)=|\phi|^2_{k}(p)e^{-\psi_t(p)} \leq |\phi|^2_{k}(p) \leq \frac{1}{2}$. Otherwise, we have $\psi_{t}(p)<0$. Then from equation \eqref{2}, we get $\Big( 2\Delta f_t + \Delta \psi_t + (1+ 2\alpha(1-t))(2r_1+1) \Big) (|\phi|^2_{g_t}-1)(p)<0$. Since $t \in I_2$, it follows that  $|\phi|^2_{g_t}<1$. 
\end{proof}
 Next, we will prove $C^0$ estimates of $\psi_t$.
\begin{lem}\label{6}
If $t\in I_{2}$, then $||\varepsilon \psi_t ||_{C^0} \leq C$, where $C$ is independent of $\varepsilon$.
\end{lem}

\begin{proof}
Suppose $\psi$ attains its maximum at a point $p$, then $\Delta \psi (p) \leq 0 $. From equation \eqref{2} and $\Delta f + r_1+ \alpha (1-t)(2r_1+1) >0$ we get
\begin{equation*}
    2(2r_1+1)(4r_2+2)(2\alpha +1)\varepsilon \psi(p) \leq (1+2\alpha)(4r_2+2)
\end{equation*}
which yields
\begin{eqnarray}\label{7}
    \varepsilon \psi(p) \leq \frac{1}{2(2r_1+1)}.
\end{eqnarray}

At a minimum point $q$ for $\psi$ we have $\Delta \psi (q)\geq 0$. Now in view of equation \eqref{1} this gives \\
\begin{equation}\label{8}
\begin{split}
    & \Big\{ \Big( \Delta f + r_1+ \alpha (1-t)(2r_1+1)\Big) \Big( (2r_2+2)- |\phi|_{g}^2 + \alpha (1-t)(4r_2+2) \Big)\\
    & \Big( 2r_2 + |\phi|_{g}^2 +\alpha (1-t)(4r_2+2) \Big) \Big\}(q) \leq \Big\{ \Big( \Delta f_0+ \Delta \psi_0 + (r_1 + 1)+ \alpha (2r_1+1) \Big)\\
    & \Big( 2r_2 + |\phi|_{g_0}^2 + \alpha (4r_2+2) \Big)\Big( (2r_2+2) - |\phi|_{g_0}^2 + \alpha (4r_2+2) \Big) \\ 
    & \Big( \Delta f_0 + r_1+ \alpha (2r_1+1)\Big) + \sqrt{-1}\frac{\nabla^{1,0} \phi \wedge \nabla^{0,1}\phi^{\dag_{g_0}}}{\omega_{\Sigma}}\\
    & \Big( 2r_2 +  |\phi|_{h_0}^2 + \alpha (4r_2+2) \Big) \Big( \Delta f_0 + r_1+ \alpha (2r_1+1)\Big)\Big\}(q),
\end{split}
\end{equation}
Applying Lemma \ref{i5} to the right-hand side term of the equation \eqref{8}, the following inequality holds for some $C>0$, independent of $\varepsilon$
\begin{equation}\label{l8}
    \begin{split}
        & \Big\{ \Big( \Delta f + r_1+ \alpha (1-t)(2r_1+1)\Big) \Big( (2r_2+2)- |\phi|_{g}^2 + \alpha (1-t)(4r_2+2) \Big)\\
    & \Big( 2r_2 + |\phi|_{g}^2 +\alpha (1-t)(4r_2+2) \Big) \Big\}(q) \leq C.
    \end{split}
\end{equation}
Equation \eqref{2} implies
\begin{equation}\label{9}
    \Big\{4(|\phi|_{g}^2-1) \Big(\Delta f + r_1 +\alpha(1-t)(2r_1+1)\Big) \Big\}(q) \leq 2(2r_1+1)(4r_2+2)(2\alpha +1)\varepsilon \psi (q)
\end{equation}
Using Lemma\ref{4}, equations \eqref{l8} and \eqref{9} we can see
\begin{eqnarray}\label{10}
 \nonumber && 2(2r_1+1)(4r_2+2)(2\alpha +1)\varepsilon \psi(q) \\
 \nonumber && \geq \frac{2C\big(|\phi|^2_{g}-1\big)}{\Big( (2r_2+2)- |\phi|_{g}^2 + \alpha (1-t)(4r_2+2) \Big) \Big( 2r_2 + |\phi|_{g}^2 + \alpha (1-t)(4r_2+2) \Big)}(q)\\
 && \geq - C.
\end{eqnarray}
Therefore, equations \eqref{7} and \eqref{10} establish the result.
\end{proof}
 Now we are in a position to prove an essential estimate, which will be used to show closedness as well as openness.
\begin{lem}\label{11}
For $t \in I_{2} $, $-\Tilde{C} \leq \Delta \psi_t \leq C$ for some positive constant $\Tilde{C}$ independent of $\varepsilon$.
\end{lem}

\begin{proof}
Since $t\in I_2$, we have $\Delta f + r_1+ \alpha (1-t)(2r_1+1) >0$. So from equations \eqref{2} we get
\begin{equation*}
    \begin{split}
        & \big(1+\Delta \psi \big)\Big( 2(|\phi|^2_{g}-1)+(1+2\alpha(1-t))(4r_2+2) \Big) \\
        & \geq 2(2r_1+1)(4r_2+2)(2\alpha +1)\varepsilon \psi_t.
    \end{split}
\end{equation*}
  Thus using Lemma\ref{6} we see $\Delta \psi \geq -\Tilde{C}$, for some positive $\Tilde{C}$ independent of $\varepsilon$. Now equation \eqref{1} gives
  \begin{equation*}
    \begin{split}
        & \big(1+\Delta \psi \big)\Big( \Delta f + r_1+ \alpha (1-t)(2r_1+1)\Big)\Big( (2r_2+2)- |\phi|_{g}^2 \\
        & + \alpha (1-t)(4r_2+2) \Big)\Big( 2r_2 + |\phi|_{g}^2 + \alpha (1-t)(4r_2+2) \Big) \leq C.
    \end{split}  
  \end{equation*}
 
  Using equation \eqref{2} we have
  \begin{equation*}
      \begin{split}
           & \big(1+\Delta \psi \big) \bigg\{ \big(1 + \Delta \psi \big) \Big( 2(|\phi|^2_{g} -1 )+ \Big ( 1 + 2\alpha (1-t) \Big) (4r_2+2) \Big)\\
           & \qquad \qquad -2(2r_1+1)(4r_2+2)(2\alpha+1))\varepsilon \psi \bigg\} \\
           & \leq \displaystyle \frac{C}{\left( (2r_2+2)- |\phi|_{g}^2 + \alpha (1-t)(4r_2+2) \right) \left( 2r_2 + |\phi|_{g}^2 + \alpha (1-t)(4r_2+2) \right)},
      \end{split}
  \end{equation*}
  which implies $\Delta \psi \leq C$.
\end{proof}
With these above estimates of $\psi_t$ in hand, we now estimate $f_t$ as follows.
\begin{lem}\label{12}
If $t \in I_2$, then $||\Delta f_t|| \leq C$.
\end{lem}

\begin{proof}
 Since $t\in I_2$, we have $\Delta f + r_1+\alpha (1-t)(2r_1+1)>0$. Therefore from the equation \eqref{1} we get
\begin{equation*}
    \Big( \Delta f +\Delta \psi +(r_1+1)+ \alpha (1-t)(2r_1+1)\Big) \Big(  \Delta f + r_1+ \alpha (1-t)(2r_1+1) \Big) \leq C
\end{equation*}
Therefore $\Delta f \leq C$, otherwise it contradicts Lemma \ref{11}.
\end{proof}

Using the formula \eqref{013} and Lemma \ref{12}, we have the following.

\begin{lem}
There exists positive constant $C$ such that $||f_t||_{C^0}\leq C$, whenever $t \in I_2$.
\end{lem}
Now one can easily conclude that $I_2$ is closed as follows.\\

Let $t_n \in I_2$ be such that $t_n \rightarrow t$. To prove $I_2$ is closed, we must show $t \in I_2$. As $t_n \in I_2$, we have $ \Delta f_{t_n}+r_1+\alpha (1-t)(2r_1+1)>0$. Arzela-Ascoli theorem, together with the above estimates, we get a subsequence of $t_n$ again call the subsequence by $t_n$ such that $f_{t_n} \rightarrow f $ and $\psi_{t_n} \rightarrow \psi$ in $C^{2,\beta}$. Then $\Delta f + r_1 + \alpha (1-t)(2r_1+1) \geq 0$ and by usual bootstrapping argument we have $f$ and $\psi$ are smooth. Now if $\Delta f + r_1 + \alpha (1-t)(2r_1+1) = 0$ at some point, then (\ref{1}) gives a contradiction. Hence $\Delta f + r_1 + \alpha (1-t)(2r_1+1) >0$. Therefore $t \in I_2$.\\

The only thing we are left with is proving the openness of $I_2$.

\subsection {Openness of \texorpdfstring{$I_2$}{\texttwoinferior}}
For $0<\beta <1$, let $\mathcal{C}$ be the subset of $C^{2,\beta} \times C^{2,\beta}$ defined by  $\mathcal{C} := \Big\{ (f, \psi) \in C^{2, \beta} \times C^{2, \beta} \Big| \Delta f + r_1 + \alpha (1-t)(2r_1+1)>0, \displaystyle \int_{\Sigma} f \omega_{\Sigma}=0 \Big\}$. Now let us define the map $T: \mathcal{C} \times [0,1] \rightarrow  C^{0,\beta}\times C^{0,\beta} $ by $T(f,\psi , t)=\Big(T_1(f, \psi, t),T_2(f, \psi, t)\Big)$, where
\begin{equation}\label{14}
\begin{split}
    & T_{1}(f,\psi,t)\\
   =&\bigg\{ \Big( \Delta f+\Delta \psi +(r_1+1)+ \alpha (1-t)(2r_1+1)\Big) \Big( 2r_2 + |\phi|_g^2 +\\
    & \alpha (1-t)(4r_2+2) \Big)  \Big( \Delta f+ r_1+ \alpha (1-t)(2r_1+1)\Big) \\
    & \Big( (2r_2+2)- |\phi|_g^2 + \alpha (1-t)(4r_2+2) \Big) + \sqrt{-1} \frac{\nabla^{1,0} \phi \wedge \nabla^{0,1} \phi^{\dag_g}}{\omega_{\Sigma}} \\
    & \Big( 2r_2 + |\phi|_g^2 + \alpha (1-t)(4r_2+2) \Big) \Big( \Delta f+ r_1+ \alpha (1-t)(2r_1+1)\Big) \bigg\} \\
    & - \bigg\{ \Big( \Delta f_0 + \Delta \psi_0 + (r_1 + 1 )+ \alpha (2r_1+1) \Big) \Big( 2r_2 + |\phi|_{g_0}^2 + \alpha (4r_2+2) \Big)\\
    &  \Big( (2r_2+2)- |\phi|_{g_0}^2 + \alpha (4r_2+2) \Big) \Big( \Delta f_0 + r_1+ \alpha (2r_1+1)\Big) \\
    & + \sqrt{-1}\frac{\nabla^{1,0} \phi \wedge \nabla^{0,1}\phi^{\dag_{g_0}}}{\omega_{\Sigma}} \Big( 2r_2 + |\phi|_{g_0}^2 + \alpha (4r_2+2) \Big)\\
    & \Big( \Delta f_0 + r_1+ \alpha (2r_1+1)\Big) \bigg\},
\end{split}
\end{equation}
and
\begin{equation}\label{15}
\begin{split}
    & T_2(f,\psi,t)\\
   =& 2\bigg( 2\Delta f+\Delta \psi +\big(1+ 2\alpha (1-t)\big)(2r_1+1) \bigg) (|\phi|_g^2 -1)\\
    & + \big( \Delta \psi +1 \big)\big(1+ 2\alpha (1-t)\big)(4r_2+2)\\
    & -2(2r_1 + 1)(4r_2 + 2)(2\alpha +1) \varepsilon \psi.
\end{split}
\end{equation}
Then the linearization of $T_1$ at a point $(f, \psi)$ will be
\begin{equation}\label{16}
\begin{split}
    & DT_{1}(f,\psi,t)[\delta f,\delta\psi]\\
   =& \Big( 2r_2 + |\phi|_g^2 + \alpha (1-t)(4r_2+2) \Big)\bigg[\Big( (2r_2+2)- |\phi|_g^2 +\\
    & \alpha (1-t)(4r_2+2)\Big) \Big( 2\Delta f+\Delta \psi +\big(1+ 2\alpha (1-t)\big)(2r_1+1) \Big) \\
    & +\sqrt{-1} \frac{\nabla^{1,0} \phi \wedge \nabla^{0,1} \phi^{\dag_g}}{\omega_{\Sigma}}\bigg]\Delta \delta f + \Big[ \Big( 2r_2 + |\phi|_g^2 +  \\
    & \alpha (1-t)(4r_2+2) \Big) \Big( \Delta f+ r_1+ \alpha (1-t)(2r_1+1)\Big) \\
    & \Big( (2r_2+2) - |\phi|_g^2 + \alpha (1-t)(4r_2+2) \Big) \Big] \Delta \delta\psi \\
    & -\Big( \Delta f+ r_1+ \alpha (1-t)(2r_1+1)\Big)\bigg[|\phi|_g^2\Big\{2(1-|\phi|_g^2)\\
    & \Big( \Delta f+\Delta \psi +(r_1+1)+ \alpha(1-t)(2r_1+1)\Big)+\\
    & \sqrt{-1}\frac{\nabla^{1,0}\phi \wedge \nabla^{0,1} \phi^{\dag_g}}{\omega_{\Sigma}}\Big\} +\sqrt{-1}\frac{\nabla^{1,0} \phi \wedge \nabla^{0,1} \phi^{\dag_g}}{\omega_{\Sigma}}\\
    & \Big( 2r_2 + |\phi|_g^2 + \alpha (1-t)(4r_2+2) \Big) \bigg]\delta\psi  \\
    & +\Big( \Delta f+ r_1 + \alpha (1-t)(2r_1+1)\Big)\\
    & \bigg[\sqrt{-1}\frac{\partial(\delta \psi) \wedge \nabla^{0,1}\phi^{\dag_g}}{\omega_{\Sigma}}+\sqrt{-1}\frac{\nabla^{1,0} \phi \wedge\phi^{\dag_g} \nabla^{0,1}(\delta\psi)}{\omega_{\Sigma}}\bigg],
\end{split}
\end{equation}
and the linearization of $T_2$ at $(f, \psi )$ will be
\begin{equation}\label{17}
\begin{split}
    & DT_{2}(f,\psi,t)[\delta f,\delta\psi]\\
 =  & 4(|\phi|_g^2-1)\Delta \delta f+ \bigg(2(|\phi|_g^2-1)+\Big(1+ 2\alpha (1-t)\Big)(4r_2+2)\bigg)\Delta \delta \psi \\
    & -\delta\psi \bigg[ 2|\phi|_g^2\bigg( 2\Delta f+\Delta \psi +\Big(1+ 2\alpha (1-t)\Big)(2r_1+1) \bigg)+\\
    & 2(2r_1 + 1)(4r_2 + 2)(2\alpha +1) \varepsilon \bigg].
\end{split}
\end{equation}
We shall show that $DT=[DT_1, DT_2]$ is an isomorphism at a point $(f_t, \psi_t)$, for $t\in I_2$. Then by the implicit function theorem for Banach manifolds, we can conclude that $I_2$ is open.

Suppose $(\delta f, \delta\psi)\in Ker(DT(f,\psi, t))$ where $t \in I_2$, then we have $DT_1(f, \psi, t)[\delta f, \delta\psi]= 0$ and $DT_2(f, \psi, t)[\delta f, \delta\psi]= 0$. Now solving $DT_2[\delta f, \delta\psi]=0$ for $\Delta \delta f$ and substituting the value in $DT_1[\delta f, \delta\psi]=0$ we get,

\begin{equation}\label{19}
\begin{split}
  & \Bigg[\Big( 2r_2 + |\phi|_g^2 + \alpha (1-t)(4r_2+2) \Big)\bigg\{\Big( (2r_2+2)- |\phi|_g^2 + \\
  & \alpha (1-t)(4r_2+2)\Big) \Big( 2\Delta f+\Delta \psi +(1+ 2\alpha (1-t))(2r_1+1) \Big) \\
  & + \sqrt{-1} \frac{\nabla^{1,0} \phi \wedge \nabla^{0,1} \phi^{\dag_g}}{\omega_{\Sigma}}\bigg\}
     \Big(2(|\phi|_g^2-1)+\left(1+ 2\alpha (1-t)\right)(4r_2+2)\Big) \\ 
  & + 4\left(1-|\phi|^2_{g}\right)  \Big( 2r_2 + |\phi|_g^2 + \alpha (1-t)(4r_2+2) \Big)\Big( \Delta f+ r_1 \\ 
  & + \alpha (1-t)(2r_1+1)\Big)\Big( (2r_2+2)- |\phi|_g^2 + \alpha (1-t)(4r_2+2) \Big)\Bigg]\Delta\delta\psi \\
 =& \delta\psi \Bigg[\Big( 2r_2 + |\phi|_g^2 + \alpha (1-t)(4r_2+2) \Big)\bigg\{\Big( (2r_2+2)- |\phi|_g^2+ \\
  & \alpha (1-t)(4r_2+2)\Big) \Big( 2\Delta f+\Delta \psi +(1+ 2\alpha (1-t))(2r_1+1) \Big)\\
  & +\sqrt{-1}\frac{\nabla^{1,0} \phi \wedge \nabla^{0,1}\phi^{\dag_g}}{\omega_{\Sigma}}\bigg\}\bigg\{2|\phi|_g^2\Big( 2\Delta f +\Delta \psi +(1+ 2\alpha (1-t))(2r_1+1) \Big)\\
  & + 2(2r_1 + 1)(4r_2 + 2)(2\alpha +1) \varepsilon \bigg\}+4\big(1-|\phi|^2_{g}\big)|\phi|_g^2\Big( \Delta f+ r_1 + \\
  & \alpha (1-t)(2r_1+1)\Big) \bigg\{2(1-|\phi|_g^2)\Big( \Delta f+\Delta \psi +(r_1+1) +\\
  & \alpha(1-t)(2r_1+1)\Big) + \sqrt{-1}\frac{\nabla^{1,0}\phi \wedge \nabla^{0,1} \phi^{\dag_g}}{\omega_{\Sigma}}\bigg\} +\sqrt{-1} \frac{\nabla^{1,0} \phi \wedge \nabla^{0,1} \phi^{\dag_g}}{\omega_{\Sigma}}\\
  & \Big( 2r_2  + |\phi|_g^2 + \alpha (1-t)(4r_2+2) \Big)\Bigg]+\Big( \Delta f+ r_1+ \alpha (1-t)(2r_1+1)\Big)\\
  & \bigg\{\sqrt{-1} \frac{\partial(\delta \psi) \wedge \nabla^{0,1}\phi^{\dag_g}}{\omega_{\Sigma}} +\sqrt{-1} \frac{\nabla^{1,0} \phi \wedge\phi^{\dag_g} \nabla^{0,1}(\delta\psi)}{\omega_{\Sigma}}\bigg\}.
\end{split}
\end{equation}
Lemma (\ref{11}) shows that for $t\in I_2$, we can choose $\varepsilon$ large enough so that
\begin{eqnarray}\label{18}
   \nonumber&&  \Big( \Delta \psi + 1 +2(2r_1 + 1)(4r_2 + 2)(2\alpha +1) \varepsilon \Big) >0.
\end{eqnarray}
Using the maximum principle on the equation (\ref{19}), we have $\delta\psi=0$. Now putting $\delta\psi=0$ in $DT_2[\delta f, \delta \psi]=0$ gives $\delta f=0$. Hence for $t\in I_2$ we get $Ker(DT(f_t, \psi_t, t))=0$. Now we shall prove that $Ker(DT^*(f_t, \psi_t, t))$ is also trivial.\\

\indent As one can imagine, computing the operator $DT^*$ will be very complicated. Since we are only interested in the kernel of $DT^*$ and we already know the kernel of $DT$, it is enough to calculate the index of the operator $DT$. To do so, let us define the following.\\

 \indent For $t\in I_2$ and $s \in [0,1]$, let $T^{s}\big( f, \psi, t \big)=\Big( T_{1}^{s}\big( f, \psi, t \big), T_{2}^{s}\big( f, \psi, t \big) \Big)$, where
\begin{equation}\label{20}
\begin{split}
 & T^{s}_{1}(f,\psi,t)\\
 =& \Big( 2r_2 + s|\phi|_g^2 + \alpha (1-t)(4r_2+2) \Big)\bigg[\Big( (2r_2+2)- s|\phi|_g^2 + \\
&\alpha (1-t)(4r_2+2)\Big) \Big( 2s\Delta f + s\Delta \psi +\big(1+ 2\alpha (1-t)\big)(2r_1+1) \Big) + \\
&\sqrt{-1} s\frac{\nabla^{1,0} \phi \wedge \nabla^{0,1} \phi^{\dag_g}}{\omega_{\Sigma}}\bigg]\Delta \delta f +  \Delta \delta\psi   \Big( 2r_2 + s|\phi|_g^2 + \alpha (1-t)(4r_2+2) \Big) \\
&\Big( s \Delta f+ r_1+ \alpha (1-t)(2r_1+1)\Big)\Big( (2r_2+2)- s|\phi|_g^2 + \alpha (1-t)(4r_2+2) \Big),
\end{split}
\end{equation}
and 
\begin{equation}\label{21}
\begin{split}
& T^{s}_2(f,\psi,t)\\
= & 4s(|\phi|_g^2-1)\Delta \delta f+ \bigg(2s(|\phi|_g^2-1)+\Big(1+ 2\alpha (1-t)\Big)(4r_2+2)\bigg)\Delta \delta \psi.
\end{split}
\end{equation}
The following lemma says that we can talk about the index of operator $T^s$.
\begin{lem}
$T^s$ are elliptic system for $s \in [0,1]$.
\end{lem}

 \begin{proof}
 Let 
 \begin{eqnarray*}
 A=\begin{bmatrix}
  A_{11} & A_{12} \\
  A_{21} & A_{22}
 \end{bmatrix}
 \end{eqnarray*}
 where,
 \begin{equation*}
 \begin{split}
  A_{11} & = \Big( 2r_2 + s|\phi|_g^2 + \alpha (1-t)(4r_2+2) \Big)\\
  &\bigg\{ \Big( (2r_2+2)- s|\phi|_g^2+ \alpha (1-t)(4r_2+2)\Big)\\
  & \Big( 2s\Delta f + s\Delta \psi +\big(1+ 2\alpha (1-t)\big)(2r_1+1) \Big)\\
  &+ s \sqrt{-1} \frac{\nabla^{1,0} \phi \wedge \nabla^{0,1} \phi^{\dag_g}}{\omega_{\Sigma}}\bigg\},\\
 A_{12} & = \Big( 2r_2 + s|\phi|_g^2 + \alpha (1-t)(4r_2+2) \Big) \\
 & \Big( s \Delta f+ r_1+ \alpha (1-t)(2r_1+1)\Big),\\
 & \Big( (2r_2+2)- s|\phi|_g^2 + \alpha (1-t)(4r_2+2) \Big),\\
 A_{21}& =4s(|\phi|_g^2-1),\\
 A_{22}& =2s(|\phi|_g^2-1)+\Big(1+ 2\alpha (1-t)\Big)(4r_2+2).
 \end{split}
 \end{equation*}
Then 
\begin{eqnarray*}
det(A) &&=
\begin{vmatrix}
A_{11} & A_{12}\\
A_{21} & A_{22}
 \end{vmatrix}\\
&& =\Big( 2r_2 + s|\phi|_g^2 + \alpha (1-t)(4r_2+2) \Big) \Bigg\{ \Big(1+ 2\alpha (1-t)\Big)\\
 && (4r_2+2) \Big[  \Big( (2r_2+2)- s|\phi|_g^2 + \alpha (1-t)(4r_2+2)\Big) \\
 && \Big( 2s\Delta f + s\Delta \psi +\big(1+ 2\alpha (1-t)\big)(2r_1+1) \Big) + s \sqrt{-1} \frac{\nabla^{1,0} \phi \wedge \nabla^{0,1} \phi^{\dag_g}}{\omega_{\Sigma}}\Big] \\
&& + 2s(|\phi|_g^2-1) \Big[ \Big( (2r_2+2)- s|\phi|_g^2 + \alpha (1-t)(4r_2+2) \Big) \\
&& \Big(s\Delta \psi +1 \big) + s \sqrt{-1} \frac{\nabla^{1,0} \phi \wedge \nabla^{0,1} \phi^{\dag_g}}{\omega_{\Sigma}} \Big] \Bigg\}
\end{eqnarray*}
\textbf{Case (a).} If
\begin{equation*}
\begin{split}
&\Big[ \Big( (2r_2+2)- s|\phi|_g^2 + \alpha (1-t)(4r_2+2) \Big) \big(s\Delta \psi +1 \big)\\
&\qquad+ s \sqrt{-1} \frac{\nabla^{1,0} \phi \wedge \nabla^{0,1} \phi^{\dag_g}}{\omega_{\Sigma}} \Big] \leq 0.
\end{split}
\end{equation*}
As $t\in I_2$ i.e., $\Delta f + r_1 + \alpha (1-t)(2r_1+1)>0$, therefore $det(A)>0$.\\ 

\textbf{Case (b).} If
\begin{equation*}
\begin{split}
& \Big[ \Big( (2r_2+2)- s|\phi|_g^2 + \alpha (1-t)(4r_2+2) \Big) \big(s\Delta \psi +1 \big) \\
& \qquad \qquad+ s \sqrt{-1} \frac{\nabla^{1,0} \phi \wedge \nabla^{0,1} \phi^{\dag_g}}{\omega_{\Sigma}} \Big] > 0.
\end{split}
\end{equation*} 
Then,
\begin{equation*}
\begin{split}
det(A)=&\Bigg[ \Big( 2r_2 + s|\phi|_g^2 + \alpha (1-t)(4r_2+2) \Big) \Big(1+ 2\alpha (1-t)\Big)(4r_2+2) \\
&\Big( (2r_2+2)- s|\phi|_g^2 + \alpha (1-t)(4r_2+2)\Big) 2 \Big(s\Delta f + r_1 + \\
& \alpha (1-t)(2r_1+1) \Big) \Bigg] + \Bigg[ \Big( 2r_2 + s|\phi|_g^2 + \alpha (1-t)(4r_2+2) \Big) \\
&\Big\{ \Big( (2r_2+2)- s|\phi|_g^2 + \alpha (1-t)(4r_2+2) \Big) \big(s\Delta \psi +1 \big) +\\
& s \sqrt{-1}\frac{\nabla^{1,0} \phi \wedge \nabla^{0,1} \phi^{\dag_g}}{\omega_{\Sigma}} \Big\} \Big\{ 2s(|\phi|_g^2-1) +\Big(1+ 2\alpha (1-t)\Big)(4r_2+2) \Big\} \Bigg] >0,
\end{split}
\end{equation*}
the first term is positive as $t \in I_2$, and the second is positive because of the assumption. Hence $det(A)>0$ for $t \in I_2$.
\end{proof} 

We see that  $T^{s}: T^{0} \simeq T^{1}$ defines a homotopy. So the index of  Fredholm operators $T^{s}$ defined by $Ind(T^{s})=dim(Ker(T^{s}))-dim(Coker(T^{s}))$ will be constant for $0 \leq s \leq 1$, in particular $Ind(T^{0})=Ind(T^{1})$. Now

\begin{equation*}
\begin{split}
    & T^{0}(f,\psi,t)\\
   =& \Bigg( \bigg\{ \Big( 2r_2 + \alpha (1-t)(4r_2+2) \Big)\Big( (2r_2+2) + \alpha (1-t)(4r_2+2)\Big) \\
    & \big(1+ 2\alpha (1-t)\big)(2r_1+1) \Delta \delta f +\Big( 2r_2 + \alpha (1-t)(4r_2+2) \Big)\\
    & \Big( r_1+ \alpha (1-t)(2r_1+1)\Big) \Big( (2r_2+2) + \alpha (1-t)(4r_2+2) \Big) \Delta \delta\psi \bigg\}, \\
    & \big(1+ 2\alpha (1-t)\big)(4r_2+2)\Delta \delta \psi \Bigg),
\end{split}
\end{equation*}
and
\begin{equation*}
\begin{split}
   & {T^{0}}^*(f,\psi,t)\\
   = & \Bigg( \bigg\{ \big(1+ 2\alpha (1-t)\big)(4r_2+2) \Delta \delta f  -\Big( 2r_2 + \alpha (1-t)(4r_2+2) \Big) \\
  & \Big( r_1+ \alpha (1-t)(2r_1+1)\Big) \Big( (2r_2+2) + \alpha (1-t)(4r_2+2) \Big) \Delta \delta\psi \bigg\}, \\
  & \Big( 2r_2 + \alpha (1-t)(4r_2+2) \Big) \Big( (2r_2+2) + \alpha (1-t)(4r_2+2)\Big) \\
  & \big(1+ 2\alpha (1-t)\big)(2r_1+1) \Delta \delta \psi \Bigg).
\end{split}
\end{equation*} 

A simple calculation gives $Ind(T^{0})=0$. Since $DT$ and $T^{1}$ have the same principal symbol, their index must be the same. Hence, $Ind(DT)=0$ and consequently $Ker(DT)=Ker(DT^*)=0$. Therefore Fredholm’s alternative implies that $DT$ is an isomorphism. This completes the proof of the theorem (\ref{dem2}).

\section*{Acknowledgements}
 I am grateful to my advisor, Vamsi Pritham Pingali, for suggesting this problem. I also thank him for the invaluable and fruitful discussion about the same. This work is supported by a scholarship from the Indian Institute of Science.

\end{document}